\documentclass[reqno,11pt]{amsart}
\usepackage[utf8]{inputenc}
\usepackage{amsmath}
\usepackage{amsthm}

\usepackage[nobysame]{amsrefs}

 \usepackage{graphicx}

\theoremstyle{plain}
\theoremstyle{definition}
\newtheorem{theorem}{Theorem}[section]
\newtheorem{lemma}[theorem]{Lemma}

\newtheorem{definition}[theorem]{Definition}
\newtheorem{corollary}[theorem]{Corollary}
\newtheorem{proposition}[theorem]{Proposition}
\newtheorem{remark}[theorem]{Remark}
\newtheorem{convention}[theorem]{Convention}

\usepackage{amssymb}
\usepackage{amsfonts}
\usepackage{amstext}
\usepackage{amsthm}
\usepackage{blindtext}
\usepackage{array}
\usepackage{tabularx}
\usepackage[english]{babel}
\usepackage{svg}
\usepackage{graphicx}
\usepackage{xcolor}
\usepackage{latexsym}
\usepackage{cite}
\usepackage{enumerate}
\usepackage{comment}
\usepackage[export]{adjustbox}
\usepackage{import}
\usepackage{xifthen}
\usepackage{pdfpages}
\usepackage{transparent}
\usepackage{hyperref}
\usepackage{mathtools}
\usepackage{theoremref}

\hypersetup{%
   colorlinks=true,
   urlbordercolor={1 1 1},
   linkcolor={black},
   raiselinks=false
}

\hypersetup{colorlinks,urlcolor=black}

\definecolor{orange}{RGB}{255,130,0}
\definecolor{maroon}{RGB}{115,0,45}
\definecolor{green}{RGB}{0,128,0}
\definecolor{lightblue}{RGB}{3,170,255}

\title[Primitivity index bounds in free groups, and the second Chebyshev function]{Primitivity index bounds in free groups, and the second Chebyshev function}

\author{Ilya Kapovich}
\address{Department of Mathematics and Statistics, Hunter College of CUNY\newline
  \indent 695 Park Ave, New York, NY 10065
  \newline \indent  {\url{http://math.hunter.cuny.edu/ilyakapo/}}}
  \email{\tt \href{mailto:ik535@hunter.cuny.edu}{ik535@hunter.cuny.edu}}

\author{Zachary Simon}
\address{Advanced Technology Center, Lockheed Martin Space\newline
  \indent 1111 Lockheed Martin Way, Sunnyvale, CA 94089}  
  \email{\tt \href{mailto:zacharygsimon@gmail.com}{zacharygsimon@gmail.com}}

\makeatletter
\@namedef{subjclassname@2020}{%
	\textup{2020} Mathematics Subject Classification}
\makeatother

\subjclass[2020]{Primary 20F65, Secondary 20F10, 20F67}
\keywords{Free groups, primitive elements, Chebyshev function}
\thanks{The first author was supported by the NSF grant DMS-1905641}

\begin{document}

\maketitle

\begin{abstract}
Motivated by results about ``untangling" closed curves on hyperbolic surfaces, Gupta and Kapovich introduced the primitivity and simplicity index functions for finitely generated free groups, $d_{prim}(g;F_N)$ and $d_{simp}(g;F_N)$, where $1\ne g\in F_N$, and obtained some upper and lower bounds for these functions. In this paper, we study the behavior of the sequence $d_{prim}(a^nb^n; F(a,b))$ as $n\to\infty$. Answering a question from~\cite{Kap19}, we prove that this sequence is unbounded and that for $n_i=lcm(1,2,\dots,i)$, we have $|d_{prim}(a^{n_i}b^{n_i}; F(a,b))-\log(n_i)|= o(\log(n_i))$. By contrast, we show that for all $n\ge 2$, one has $d_{simp}(a^nb^n; F(a,b))=2$.  In addition to topological and group-theoretic arguments, number-theoretic considerations, particularly the use of asymptotic properties of the second Chebyshev function, turn out to play a key role in the proofs.
\end{abstract}

\section{Introduction}

In recent years the study of quantitative aspects of residual finiteness for various classes of finitely generated groups has become an active theme in geometric group theory. See \cite{GK14, BBRKM17, BR10, BRC16, BRHP15, BRK12, BRMR10, BRMR11, BRMR15, BRM17, BRSe16, BRSt16, Bu09, GV10, KM11, KT16, Pa14, Ri12}.
The topic is closely related to topological and geometric results about ``untangling" closed curves on hyperbolic surfaces. A classic result of Scott in \cite{Sc85} from the 1980s showed that if $\Sigma$ is a closed hyperbolic surface with a hyperbolic metric $\rho$, and $\gamma$ is an essential closed geodesic on $\Sigma$, then $\gamma$ lifts to a simple closed geodesic $\widehat \gamma$ in some finite cover $\widehat \Sigma$ of $\Sigma$. Scott's proof exploits subgroup separability of the fundamental group $\pi_1(\Sigma)$ of $\Sigma$, which is a stronger form of residual finiteness. More recently, Patel \cite{Pa14} proved that in the context of Scott's theorem, one can bound the degree $d$ of the cover $\widehat \Sigma$ of $\Sigma$ from above by $C\ell_\rho(\gamma)$, where $C=C(\Sigma,\rho)>0$ is some constant independent of $\gamma$. One can then define the \emph{untangling degree} $\deg_{\Sigma,\rho}(\gamma)$ as the smallest degree $d$ of a finite cover of $\Sigma$ to which $\gamma$ lifts or ``untangles" as a closed geodesic. Using this quantity, one then defines the ``worst-case'' function $f_{\Sigma,\rho}(L)$ as the maximum of $\deg_{\Sigma,\rho}(\gamma)$ where $\gamma$ varies over all essential closed geodesics of length $\le L$. (Here $L$ needs to be assumed $\ge sys(\Sigma,\rho)$, the length of the shortest essential closed geodesic on $(\Sigma,\rho)$.) Patel's result can now be restated as saying that  $f_{\Sigma,\rho}(L)\le CL$ for all $L\ge sys(\Sigma,\rho)$.  Similar inequalities, for similarly defined quantities, hold for more general types of finite type hyperbolic surfaces. Moreover, a simple closed curve on a surface is a special case of a non-filling closed curve. Thus, stated again for a closed hyperbolic surface $(\Sigma,\rho)$, and an essential closed geodesic $\gamma$ on $\Sigma$, one can define $\deg_{\Sigma,\rho}^{fill}(\gamma)$ as the smallest degree of a finite cover $\widehat \Sigma$ of $\Sigma$ to which $\gamma$ lifts as a non-filling curve in $\widehat \Sigma$. This notion leads to a similarly defined worst-case function $f_{\Sigma,\rho}^{fill}(L)$. By definition, one has $\deg_{\Sigma,\rho}^{fill}(\gamma)\le \deg_{\Sigma,\rho}(\gamma)$ and $f_{\Sigma,\rho}^{fill}(L)\le f_{\Sigma,\rho}^{fill}(L)$. These quantities were formally introduced in \cite{GK14}, and we refer the reader there for a more detailed discussion.

Motivated by the case of hyperbolic surfaces, Gupta and Kapovich \cite{GK14} introduced similar notions for finite rank free groups $F_N=F(A)$, where $N\ge 2$ and $A=\{a_1,\dots, a_N\}$. For $g\in F_N$, we denote by $|g|_A$ and by $||g||_A$ the freely reduced length and the cyclically reduced length of $g$ with respect to $A$ accordingly.

Marshall Hall's theorem in \cite{Ha49} easily implies that for every $1\ne g\in F_N$, there exists a subgroup $H\le F_N$ of finite index such that $g\in H$ and such that $g$ is \emph{primitive} in $H$, that is, $g$ belongs to some free basis of $H$. Moreover, the Stallings subgroup graphs proof in \cite{St83} of the Marshall Hall Theorem implies that one can always find such an $H$ with $[F_N:H]\le ||g||_A$. In a nonabelian free group $U$, a primitive element is a special example of a ``simple element.'' Here an element $1\ne g\in U$ is called \emph{simple} if there exists a free product decomposition $U=U_1\ast U_2$ with $U_1\ne 1, U_2\ne 1$ such that $g\in U_1$. For $1\ne g\in F_N$ one then defines the \emph{primitivity index} $d_{prim}(g;F_N)$ as the smallest index $[F_N:H]$ of a subgroup $H\le F_N$ such that $g\in H$ and that $g$ is primitive in $H$. Similarly, for $1\ne g\in F_N$ one defines the \emph{simplicity index} $d_{simp}(g;F_N)$ as the smallest index  $[F_N:H]$ of a subgroup $H\le F_N$ such that $g\in H$ and that $g$ is simple in $H$. Using these indices,  \cite{GK14} then defined the corresponding worst-case functions, the \emph{primitivity index function} $f_{prim}(n;F_N)$ and the the \emph{simplicity index function} $f_{prim}(n;F_N)$. We discuss  some properties of these functions further below. In particular, as shown in \cite{GK14}, for every $1\ne g\in F_N=F(A)$, one has 
\[
d_{simp}(g;F_N)\le d_{prim}(g;F_N)\le ||g||_A\le |g|_A.
\]
In the appendix to \cite{GK14},  deploying a connection with the residual finiteness growth function for $F_N$, Bou-Rabee obtained a lower bound for $f_{prim}(n;F_N)$ that grows essentially as $n^{1/4}$. Moreover, he showed that modulo a conjecture of Babai in finite group theory, one gets a lower bound for $f_{prim}(n;F_N)$ that is slightly sublinear in $n$. Gupta and Kapovich also obtained a lower bound of $C\frac{\log(n)}{\log \log(n)}$ as $n\to\infty$ for $d_{simp}(n,F_N)$. 

These bounds rely on highly indirect non-constructive  arguments. In practice, understanding the properties of  $d_{prim}(g_n;F_N)$ for explicit sequences of elements $g_n\in F_N$ with $||g_n||_A$ growing linearly in $n$ is quite hard, and in the examples that have been analyzed $d_{prim}(g_n;F_N)$ is either bounded above by a constant or has linear growth in $n$ itself. In particular, there have been no known examples of this type where $d_{prim}(g_n;F_N)$ is an unbounded sequence that grows sublinearly.  

In the present paper, we produce the first example of a sequence of elements in $F_2=F(a,b)$ that exhibits such new behavior. The main family of words we consider in this paper is $w_n=a^n b^n\in F_2=F(a,b)$ where $n\ge 1$. 

For this family, we obtain the following bounds from Theorem \ref{thm1} and Theorem \ref{thm2}, respectively:

\begin{theorem}\label{t:main2}
There exists a constant $C\ge 0$ such that the following hold:
\begin{enumerate}
\item[(a)] For all integers $n\ge 1$, we have
\[
d_{prim}(a^n b^n; F_2)\le \log(n)+C.
\]

\item[(b)] For all integers $i\ge 1$, put $n_i=lcm(1,2,\dots,i)$. Then for all $i\ge 1$, we have
\[
d_{prim}(a^{n_i} b^{n_i}; F_2)\ge \log(n_i)-o(\log(n_i)).
\]
\end{enumerate}
\end{theorem}

Theorem~\ref{t:main2} directly implies the following:

\begin{corollary}\label{cor:main3}
Let $C\ge 0$ be the constant provided by Theorem~\ref{t:main2}. For $i=1,2,3,\dots$, put $n_i=lcm(1,2,3,\dots, i)$.

Then for all $i\ge 1$, we have
\[
\log(n_{i})-o(\log(n_i))\le d_{prim}(a^{n_i} b^{n_i}; F_2)\le \log(n_{i})+C.
\]
\end{corollary}

Corollary~\ref{cor:main3} answers, in the negative, the question raised in \cite{Kap19} as to whether the sequence $d_{prim}(a^nb^n;F_2), n\ge 1$ is bounded. Corollary~\ref{cor:main3} shows that for the sequence $w_{n_i}=a^{n_i}b^{n_i}\in F(a,b)$ as above with $||w_{n_i}||=2n_i$, we have $\big|d_{prim}(w_{n_i}; F_2)-\log(n_i)\big|= o(\log(n_i))$. This result provides the first explicit example of a sequence of cyclically reduced words whose length grows linearly but whose primitivity index function is unbounded and sublinear. Moreover, in this situation $d_{prim}(w_{n_i}; F_2)$ is computed almost precisely, up to a relatively small additive error, which in earlier known examples only happened in rather trivial cases.

By contrast, it turns out that the sequence $d_{simp}(w_n, F_2)$ is bounded and in fact constant:

\begin{theorem}\label{t:main4}
For all integers $n\ge 2$, we have
\[
d_{simp}(a^n b^n; F_2)=2.
\]
\end{theorem}

Theorem~4.14 in \cite{GK14} provides an algorithm for computing $d_{prim}(g)$ and $d_{simp}(g)$ for $1\ne g\in F_N$. However, that algorithm involves some costly enumeration procedures that make it non-practical. Moreover, the main results of \cite{GK14} suggest that precisely computing $d_{prim}(g)$ and $d_{simp}(g)$ is difficult even for "random" elements in $F_N$. Thus computing $d_{prim}(g)$ and $d_{simp}(g)$ is generally difficult in practice, except for some special algebraic circumstances. For example, with a bit of work one can show directly that $d_{prim}(a^3 b^3; F_2)=3$. However, say, computing $d_{prim}(a^5 b^5; F_2)$ already appears to be hard to do by hand. Obtaining more precise information about $d_{prim}(a^nb^n; F_2)$ than that provided by Theorem~\ref{t:main2} also appears to be a difficult but interesting task.

As noted above, most previous proofs, both for free groups and for surfaces, for lower bounds of the index and degree functions of the type discussed in this paper involved rather indirect and implicit arguments. The one exception was provided by a paper of Gaster \cite{Ga16} where he used an explicit sequence of curves $\gamma_n$ on $\Sigma$ to prove that $\displaystyle f_{\Sigma,\rho}(L)\ge_{L\to\infty} c_0 L$.

The proofs of the main results in this paper deploy a combination of topological, group-theoretic, and number-theoretic methods. The connection with number theory comes from the following fact, see Lemma~\ref{lem2} below, whose proof uses basic known properties of the second Chebyshev function. Let $n\ge 3$ be an integer and let $d=d(n)\ge 2$ be the smallest positive integer such that $d\nmid n$. Then $d(n)\le \log(n)+C$ for some constant $C$. Moreover, if $n_i=lcm(1,\dots, i)$ then $d(n_i)\ge  \log(n_i)-o(\log(n_i))$. 
 
For the proof of the upper bound in part (a) of  Theorem~\ref{t:main2}  we construct an explicit subgroup $H$ of index $d(n)$ in $F(a,b)$ containing $w_n=a^nb^n$ and verify that that $w_n$ is primitive in $H$. (The subgroup $H$ is the kernel of an epimorphism from $F(a,b)$ onto the cyclic group $\mathbb Z_d$.) Hence, $d_{prim}(w_n; F_2)\le d(n)\le \log(n)+C$.
 
The proof of the lower bound for $d_{prim}(w_{n_i})$  in part (b) of  Theorem~\ref{t:main2} is more involved. The main algebraic trick is Lemma~\ref{power}. It shows that if $H$ is a subgroup of finite index in $F_2=F(a,b)$ and $a^k, b^l$ are the smallest positive powers of $a,b$ that belong to $H$ then there exists a free basis of $H$ containing both $a^k$ and $b^l$. We take $d=d(n_i)>1$ to be the smallest positive integer such that $d\nmid n_i$ and that $H$ is a subgroup of $F_2$ of index $m<d$ containing $w_{n_i}$. Then $a^k,b^l$ chosen as above satisfy $k,l\le m<d<n_i$. The definition  of $d$ implies that $k|n_i, l|n_i$ and therefore $w_{n_i}=a^{n_i}b^{n_i}=(a^k)^p(b^l)^q$ with $p,q\ge 2$. Since $a^k,b^l$ belong to a common free basis of $H$, a standard Whitehead graph argument implies that $w_{n_i}$ is not primitive in $H$. Therefore, by definition of $d_{prim}$, we have $d_{prim}(w_{n_i})\ge d(n_i)$. Well-known number-theoretic facts about the second Chebyshev function then imply that $d(n_i)\ge \log(n_i)-o(\log(n_i))$, and part (b) of  Theorem~\ref{t:main2} follows. 
 
We also obtain (see Proposition~\ref{prop4} below) the following upper bound result for words $a^nb^t\in F_2$ where $n,t\ge 1$ are arbitrary and not necessarily equal integers.  
\begin{theorem}
Let $n,t\ge 1$ and let $d,d'\ge 2$ be integers such that $d\nmid n$ and $d'\nmid t$, and that $d\le n, d'\le t$. Then  
\[
d_{simp}(a^n b^t; F_2)\le d_{prim}(a^n b^t; F_2)\le d+d'-2.
\]
\end{theorem}

Note that the true asymptotics of $f_{simp}(n;F_N)$ and of $f_{\Sigma,\rho}^{fill}(L)$ remain a mystery. The results of Gupta and Kapovich \cite{GK14} provide only a fairly weak $C\frac{\log(n)}{\log \log(n)}$ lower bound for $d_{simp}(n,F_N)$. Gaster's proof \cite{Ga16} of a linear lower bound for $f_{\Sigma,\rho}(L)$ uses a sequence of curves $\gamma_n$ on $\Sigma$ that are non-filling, and thus have $\deg_{\Sigma,\rho}^{fill}(\gamma_n)=1$. Therefore, his argument sheds no light on the behavior of $f_{\Sigma,\rho}^{fill}(L)$.  

The results of the present paper indicate that using explicit sequences of group elements and curves may provide a fruitful approach to better understanding the behavior of $f_{simp}(n; F_N)$ and $f_{prim}(n; F_N)$ for free groups and of $f_\Sigma^{fill}(L)$ for surfaces.  

{\bf Acknowledgements.}

We are grateful to the initial referee for spotting an error in a number-theoretic estimate in the original version of this paper, and also for suggesting simplifications to our original proof of Theorem~\ref{t:main2}.

The first author was supported by the NSF grant DMS-1905641.

\section{Preliminaries}\label{S:prelim}

\subsection{Graphs}

We will use the notations and terminology regarding graphs, $A$-graphs, folded $A$-graphs, Stallings folds, etc., from \cite{KM02,GK14,St83}, and we refer the reader for the details to those sources. We briefly recall some of the relevant definitions here.

\begin{definition}\label{def1}
A  \emph{graph} $\Gamma$ is a 1-dimensional cell-complex. The 0-cells of $\Gamma$ are called \emph{vertices} and the set of vertices of a graph $\Gamma$ is labeled as $V \Gamma$.

Taking open 1-cells, \emph{topological edges} of $\Gamma$, these are homeomorphic to the open unit interval, $(0, 1)$, which is a 1-manifold having two orientations. An \emph{oriented edge} is a topological edge endowed with an orientation. For an oriented edge $e$, we denote by $\bar e$ the same topological edge with the opposite orientation. Note that for an oriented edge $e$ of $\Gamma$, we always have $e\ne \bar e$ and $\bar{\bar e}=e$.

We denote by $E \Gamma$  the set of oriented edges of a graph $\Gamma$.

Due to the fact that $\Gamma$ is a cell-complex, every oriented edge is endowed with some orientation-preserving map $j_{e}: [0, 1] \rightarrow \Gamma$, which provides a homeomorphism between the open unit interval $(0, 1)$ and an edge $e$ such that $j_{e}(0), j_{e}(1) \in V \Gamma$. And for any edge in the edge set, accordingly denote $j_{e}(0)$ and $j_{e}(1)$ by $o(e)$ and $t(e)$, which correspond to \emph{initial} and \emph{terminal} vertices of \emph{e}, respectively.

For a vertex $v\in V\Gamma$, the \emph{degree} $\deg_\Gamma(v)$ of $v$ in $\Gamma$ is the cardinality of the set $\{e\in E\Gamma| o(e)=v\}$. 

For all $i$, denote a sequence of edges $(e_i)_{i = 1}^{i = k}$, such that $(e_i) \in E \Gamma$, as an \emph{edge-path} $p \in \Gamma$ where $o(e_j) = t(e_{j - 1})$ for all $2 \leq j \leq k$. The length of the path \emph{p}, $|p|$, is defined as the number of edges in $p$. A \emph{reduced path} is a path that has no subpaths with cancellations from an edge and its inverse. Also, the set of reduced edge-paths from $x$ to $x$, for some $x \in V \Gamma$, will be identified as the fundamental group $\pi_1 (\Gamma, x)$.
\end{definition}

\begin{definition}[$A$-graph]
For an integer $N\ge 2$, denote by $F_N=F(a_1,\dots,a_N)$ the \emph{free group} of rank $N$ with the \emph{free basis} $A = \{ a_1,\dots,a_N \}$. 

An \emph{$A$-graph} is a graph $\Gamma$ together with the \emph{labelling map} $\mu:E\Gamma\to A\cup A^{-1}$ such that for every $e\in E\Gamma$, we have $\mu(\bar e)=(\mu(e))^{-1}$. 

An $A$-graph $\Gamma$ is \emph{folded} if there do not exist a vertex $x\in V\Gamma$  and edges $e_1,e_2\in E\Gamma$ with \\ $x = o(e_1) = o(e_2)$ such that $e_1\ne e_2$ and $\mu(e_1) = \mu(e_2)$.

The \emph{N-rose} $R_N$ is the wedge of $N$ loop-edges labelled at vertex $v_0$ consisting of edges $a_1, \dots, a_N$. Thus, $R_N$ is a folded $A$-graph.
\end{definition}

Note that there is a natural identification $F_N=F(A)=\pi_1(R_N,x_0)$.  If $\Gamma$ is an $A$-graph, the edge-labeling $\mu$ canonically defines a label-respecting map $f: \Gamma\to R_N$ that sends all vertices of $\Gamma$ to $v_0$. This map $f$ is an immersion if and only if $\Gamma$ is folded.
Moreover, if $\Gamma$ is folded, the corresponding map $f: \Gamma\to R_N$ is a covering map if and only if the graph $\Gamma$ is $2N$-regular, and in this case the degree of the covering is equal to $\#V\Gamma$. 

\subsection{Primitive and Simple Words}
\begin{definition}\label{def2}
A nontrivial element $w \in F_N$ is called \emph{primitive} in $F_N$ if \emph{w} belongs to a free basis of $F_N$.

A nontrivial element $w \in F_N$ is called \emph{simple} in $F_N$ if \emph{w} belongs to a proper free factor of $F_N$.

The \textit{primitivity index} $d_{prim}(w) = d_{prim}(w; F_N)$ of $w \in F_N$ is the smallest possible index for a subgroup $H\le F_N$ containing \emph{w} as a primitive word.

The \textit{simplicity index} $d_{simp}(w) = d_{simp}(w; F_N)$ of $w \in F_N$ is the smallest possible index for a subgroup $H\le F_N$ containing \emph{w} as a simple word \cite{GK14}.
\end{definition}

\begin{remark}\label{rmk1}
If $w \in F_N$ is primitive, then $w$ is also simple in $F_N$. As discussed in the Introduction,  for every $1\ne g\in F_N=F(A)$, one has~ \cite{GK14}:
\[
d_{simp}(g)\le d_{prim}(g)<||g||_A\le |g|_A<\infty.
\]

Note that the primitivity and simplicity of elements of $F_N$ are preserved under arbitrary automorphisms of $F_N$. Similarly, the definitions imply that for a nontrivial element of $F_N$ its primitivity and simplicity indexes are preserved by automorphisms of $F_N$ as well.
\end{remark}

\begin{remark}\label{rmk2}
Let $F_N = F(a_1, a_2, \dots, a_N)$. Let $w\in F(a_1, a_2, \dots, a_N)$ be a freely reduced word such that for some $1\le i\le N$ the generator $y_i^{\pm 1}$ appears in $w$ exactly once.  Then $w$ is primitive in $F_N$.
\end{remark}

\begin{proposition} \cite[Lemma 3.6]{GK14} \label{rmk3}
Let $N\ge 2$. Then for all integers $n\ge 1$ we have \[f_{simp}(n; F_N)\le f_{prim}(n; F_N)\le n.\]
\end{proposition}

We recall  how simple and primitive words are related to Whitehead graphs due to work that Stallings established \cite{St99} by generalizing results from Whitehead \cite{Wh36}. We refer the reader to \cite{GK14, KSS06} for additional references and further background information on Whitehead graphs.

\begin{definition}[Whitehead graph]\label{def4}
Let $F_N=F(x_1,\dots, x_N)$ be the free group of finite rank $N\ge 2$ and let $w \in F_N$ be a nontrivial cyclically reduced word. Let $c$ be the first letter of $w$, so that the word $wc$ is freely reduced. We now define the \emph{Whitehead graph} of $w$, denoted $\Gamma_w$, as a simple graph with vertex set $V\Gamma_w=\{ x_{1}^{\pm1}, \dots, x_{N}^{\pm1} \}$ and with the edge set defined as follows. 

For $x,y\in V\Gamma_w$ such that $x^{-1}\ne y$, there exists an undirected edge $\{x^{-1}, y\}$ in $\Gamma_w$ joining $x^{-1}$ and $y$ whenever $xy$ or $y^{-1} x^{-1}$ occurs as a subword of $wc$.
\end{definition}

\begin{definition}\label{def5}
A \emph{cut vertex} in a graph $\Gamma_w$ is a vertex \emph{x} such that $\Gamma_w - \{ x \}$ is disconnected.
\end{definition}
 Note that if $\Gamma_w$ has at least one edge and is disconnected, then $\Gamma_w$ has a cut vertex, i.e., any end-vertex of an edge of $\Gamma_w$ is a cut vertex \cite{GK14}.

We will need the following important result of Stallings~\cite{St99} about Whitehead graphs of simple elements (this result was proved earlier by Whitehead~\cite{Wh36} for primitive elements):
\begin{proposition}\label{prop5}
If $w \in F_N$ is a simple and nontrivial cyclically reduced word, then $\Gamma_w$ has a cut vertex.
\end{proposition}

\begin{corollary}\label{cor2}
Let $F_N=F(a_1,\dots, a_N)$ be the free group of finite rank $N\ge 2$. Let $k_1,\dots, k_N\ge 2$ be arbitrary integers and let $w=a_1^{k_1} \cdots a_N^{k_N} \in F_N$. Then $w$ is not simple (and in particular, not primitive) in $F_N$.
\end{corollary}

\begin{proof}
Let $k_1,\dots, k_N\ge 2$ and let $w=a_1^{k_1} \cdots a_N^{k_N} \in F_N=F(a_1,\dots, a_N)$. Thus, $w$ is a nontrivial freely and cyclically reduced word. We now construct the Whitehead graph $\Gamma_w$ as defined in Definition \ref{def4}.

The two-letter subwords cyclically occurring in $w$ are precisely $a_i^2$, where $i=1,\dots, N$ and $a_ia_{i+1}$ where $i=1,\dots, N-1$, as well as the subword $a_Na_1$. Therefore, as in Figure \ref{WHGraph}, the edges in the (simple) graph $\Gamma_w$ are as follows:

For $i=1,\dots, N$ we have an edge $\{a_i, a_i^{-1}\}$. For $i=1,\dots, N-1$ we have an edge $\{a_i^{-1}, a_{i+1}\}$, and we also have an edge $\{a_N^{-1},a_1\}$. Thus, we see that the graph $\Gamma_w$ is a topological circle with the vertex set $\{a_1^{\pm 1}, \dots , a_N^{\pm 1}\}$. In particular, $\Gamma_w$ has no cut-vertices. Hence, by Proposition~\ref{prop5}, the element $w\in F_N$ is not simple. 
\end{proof}

\begin{figure}
    \centering
    \includegraphics[width=3in]{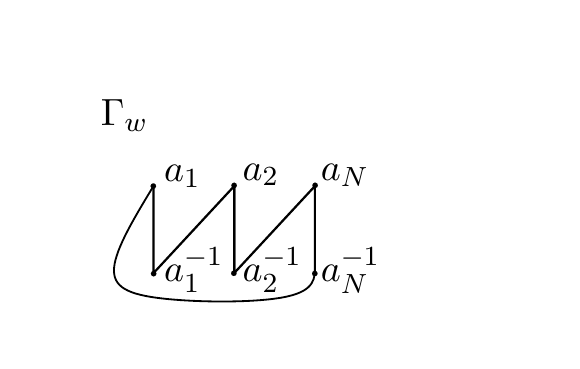}
    \caption{Whitehead Graph}
    \label{WHGraph}
\end{figure}

We recall the following useful fact about primitivity in free groups, see~\cite[Proposition 4.5]{GK14}:

\begin{proposition}\label{p:pr}
Let $F_N$ be a free group of finite rank $N\ge 2$, let $U\le F_N$ be a free factor of $F_N$ and let $1\ne g\in U$. Then $g$ is primitive in $U$ if and only if $g$ is primitive in $F_N$.
\end{proposition}

\subsection{Second Chebyshev Function}\label{Chebyshev}
Before proving our main results, we turn to a discussion on the \emph{second Chebyshev function}, $\psi(m)$, where for an integer $m\ge 1$ we have $e^{\psi(m)}=lcm(1,2,\dots,m)$. We will repeatedly deploy the asymptotics of $\psi(m)$ to find primitivity and simplicity bounds in our main results. Taking logarithms, one gets $\psi(m)=\log[lcm(1,2,\dots,m)]$. Historically, there has been a great deal of research on analyzing the growth rate of the second Chebyshev function, and its properties are closely related to the prime-counting function and the Prime Number Theorem.

A well-known result concerning the second Chebyshev function comes from the work of Rosser and Schoenfeld \cite{RS62}:

\begin{proposition}\label{prop1} \cite[Theorem 11]{RS62}.
Let $R=\frac{515}{(\sqrt{546}-\sqrt{322})^2} \approx 17.51631$ and \[\varepsilon(m)=\sqrt{\log(m)}\, exp[- \sqrt{\frac{\log(m)}{R}}].\] 
Then for $m\ge 2$ we have
\[
[1-\varepsilon(m)]m<\psi(m),
\]
and for $m\ge 1$ we have 
\[
\psi(m)<[1+\varepsilon(m)]m.
\]
\end{proposition}

\begin{proposition}\label{prop2} \cite[Theorem 12]{RS62}.
The quotient $\frac{\psi(m)}{m}$ takes its maximum at $m=113$, and for $m>0$,
\[
\psi(m)<1.03883m.
\]
\end{proposition}

Also, for primes $p$, and positive integers $k$, one has~\cite{Pe85}
\[
\psi(m) = \sum_{p^{k} \leq m} \log(p).
\]

For completeness, we prove the following well-known result in number theory that we will need in this paper:

\begin{corollary}\label{prop3}
For any $\varepsilon>0$, there exists $m_0=m_0(\varepsilon)$ such that for all $m\ge m_0$, we have
\begin{equation}\label{eq:A0}
|\psi(m)-m|<\varepsilon m
\end{equation}
\end{corollary}

\begin{proof}

Proposition \ref{prop1} implies that $\psi(m)=m+o(m)$ as $m\to\infty$. Therefore, $|\psi(m)-m|<\varepsilon m$ for all sufficiently large $m$, so that \eqref{eq:A0} holds.
\end{proof}

Note that as in Corollary~\ref{prop3}, the second Chebyshev function can be expressed as:
\[
\psi(m)=m+o(m), m \rightarrow+\infty.
\]

\begin{convention}\label{con1}
For an integer $n\ge 1$, we denote by $d(n)$ the smallest integer $d\ge 2$ such that $d\nmid n$.
\end{convention}

\begin{lemma}\label{lem1}
If $n\ge 3$, then $1<d(n)<n$.
\end{lemma}

\begin{proof}
Let $n\ge 3$. We claim that $n-1\nmid n$. Indeed suppose that $(n-1) | n$. Then $n=k(n-1)$ for $k\ge 2$, and $n\ge 2(n-1)=2n-2$ implies that $n\le 2$, which is a contradiction. Thus, $n-1\nmid n$, and hence $d(n)\le n-1<n$, as required.  
\end{proof}

To later determine bounds on the number of vertices of some graph, we need the following lemma:

\begin{lemma}\label{lem2}
There exist a constants $C'\ge 0$, an integer $n_0\ge 3$ and a function $\alpha(x)\ge 0$, $\alpha(x)=_{x\to\infty} o(x)$ with the following properties.

Let $n\ge 2$ be an integer and let $d=d(n)\ge 2$ be the smallest integer that does not divide $n$.  For $i=1,2,3,\dots $, put $n_i=lcm(1,2,\dots,i)$; in this case we use $d=d(n_i)$.

Then:

\begin{enumerate}
\item[(a)] For all $n\ge n_0$, we have $d\le \log(n)+\log(2)+1$.
\item[(b)] For all $n\ge 2$, we have $d\le \log(n)+C'$.
\item[(c)] For all $i\ge 2$, we have $d\ge \log(n_i)-\alpha(\log(n_i))$.
\end{enumerate}
\end{lemma}

\begin{proof}

We first establish part (b).
Let $n\ge 2$ and let $d\ge 2$ be the smallest integer such that $d\nmid n$.  Then for $i=1,\dots, d-1$, we have $i|n$, and hence $lcm(1,2,\dots,d-1) | n$. Therefore, $lcm(1,2,\dots,d-1) \leq n$. Denote $m=d-1$ and let $\varepsilon=\frac{1}{2} \in (0, 1)$. Thus, $lcm(1,2,\dots, m)\le n$.
For the second Chebyshev function $\psi(m)=\log[lcm(1,2,\dots,m)]$, Corollary~\ref{prop3}  implies that there exists an integer $m_0\ge 1$ such that for all $m\ge m_0$, we have $\log[lcm(1,2,\dots,m)]\ge m+\log(1-\varepsilon)=m-\log(2)$. Choose an integer $n_0\ge 1$ such that $m_0\le \log(n_0)$.

We proceed by breaking into two cases. 

First, suppose that $m=m(n)\ge m_0$. 
Then \[m-\log(2)\le \log[lcm(1,2,\dots,m)]\le \log(n).\]
Hence, $m\le \log(n)+\log(2)$. Since $m=d-1$, it follows that $d\le \log(n)+\log(2)+1$. 

Suppose now that $m=m(n)\le m_0$. Then $d=m+1\le m_0+1$.

That in both cases for all $n\ge 2$, we have $d\le \log(n)+\log(2)+1+m_0$. Thus, part (b) is established with $C'=\log(2)+1+m_0$.

We now establish part (a). Assume now that $n\ge n_0$. If $m=m(n)\ge m_0$, then we have $d\le \log(n)+\log(2) +1$ by the argument above, as required. Thus, suppose that $m<m_0$. Hence, $d=m+1\le m_0$. Recall that $n_0$ was chosen so that $m_0\le \log(n_0)$. Thus, in this case
\[
d\le m_0\le \log(n_0)\le \log(n)\le \log(n)+\log(2)+1.
\]
Hence, the conclusion of part (a) is established, as required.

Now let $i\ge 2$ and let $n_i=lcm(1,2,\dots, i)$. Let $d=d(n_i)\ge 2$ be the smallest integer such that $d\nmid n_i$. Since $1,2,\dots,i | n_i$ and $n_i=lcm(1,2,\dots,i)$, it follows that $d\ge i+1$.

Corollary~\ref{prop3}  implies that $\log(n_i)=\log[lcm(1,2,\dots,i)]=i+o(i)$. In particular, for all sufficiently large $i$, we have
\[
\frac{i}{2}\le \log(n_i) \le 2i.
\]
Therefore, $o(i)=o(\log(n_i))$ and $\log(n_i)=i+o(\log(n_i))$. Hence, 
\[
d\ge i+1\ge i=\log(n_i) - o(\log(n_i)),
\] 
and part (c) holds, as required.

\end{proof}

\section{Main Results}
Let $F_N=F(a_1,\dots,a_N)$ be the free group consisting of $N\ge2$ generators with the free basis $A=\{a_1,\dots,a_N\}$. For the remainder of this section, one of the primary objects under investigation will involve the free group $F_2$ where $N=2$. 
In this case, we denote $A=\{a,b\}$ and $F_2=F(a,b)=F(A)$.

Let $R_2$ be the 2-rose, that is an $A$-graph with a single vertex $v_0$ and two positively oriented petal-edges at $v_0$ labelled $a$ and $b$ accordingly. Then there is a natural identification $F(a,b)=\pi_1(R_2,v_0)$, and finite index subgroups of $F(a,b)$ correspond to finite connected basepointed covers of $R_2$. That is, every subgroup $H\le F(a,b)$ of finite index $q$ is uniquely represented by a $q$-fold cover $f:(\Gamma,x_0)\to (R_2,v_0)$ where $\Gamma$ is a finite connected folded 4-valent $A$-graph. In this case, we have an isomorphism $f_\#: \pi_1(\Gamma,x_0)\to H\le \pi_1(R_2,v_0)$ given by reading the labels of closed paths in $\Gamma$ at $x_0$. 

Recall that if $\Gamma$ is a finite connected $A$-graph with a base-vertex $x_0$ and $T$ is a maximal subtree of $\Gamma$ then $T$ defines a \emph{dual} free basis $S_T$ of $\pi_1(\Gamma,x_0)$ as follows.  Let $E'$ be the set of those oriented edges of $\Gamma-T$ that are labeled by elements of $A$ (rather than of $A^{-1}$). For each $e\in E'$ put $\beta_e=[x_0,o(e)]_Te[t(e),x_0]_T$. Then $S_T=\{\beta_e|e\in E'\}$. Note that if $\Gamma$ is folded then $\mu(S_T)$ is a free basis of the subgroup $H$ of $F(A)$ represented by $(\Gamma,x_0)$; this basis is also referred to as \emph{dual} to $T$.  See \cite[Section 6]{KM02} for more details.

The following lemma was suggested to us by the referee as for simplifying of our original, more topological, argument for proving part (1) of Theorem~\ref{t:main2}.

\begin{lemma}\label{lem:1}
Let $d\ge 2$ be an integer. 

(1) There exists a subgroup $H\le F_2$ with $[F_2:H]=d$ such that $H$ admits a free basis $Y=\{y_0,\dots, y_{d}\}$ where $y_0=a^d$, $y_d=b^d$ and $y_i=a^ib^i$ for $i=1,\dots, d-1$. 

(2) The subgroup $H$ from part (1) is equal to the kernel of the homomorphism $\phi:F(a,b)\to \mathbb Z_d$ given by $\phi(a)=[1]_d$ and $\phi(b)=[-1]_d$.
\end{lemma}
\begin{proof}
Take two simplicial cycles of length $d$ in the plane. Call one cycle $\Delta_1$ with edges labelled by $a$ flowing counterclockwise, and denote the other cycle by $\Delta_2$ with edges labelled by $b$ flowing clockwise.  We then superimpose $\Delta_2$ on $\Delta_1$ by a Euclidean translation and identify their vertex sets. This process results in a graph $\Gamma$ as in Figure \ref{fig2}.  Thus, $\Gamma$ is a folded connected $A$-graph with $d$ vertices with the property that for every two vertices, $v$ and $v'$, with an edge labelled by $a$ from $v$ to $v'$, there is an edge labelled by $b$ going from $v'$ to $a$ in $\Gamma$.  We still denote the (embedded) images of the $d$-cycles $\Delta_1$  and $\Delta_2$ in $\Gamma$ by $\Delta_1$  and $\Delta_2$. 
We mark one vertex $x_0$ of $\Gamma$ as a base-vertex, which defines a basepointed immersion $f:(\Gamma,x_0)\to (R_2,v_0)$. Note that $\Gamma$ is 4-regular, so that $f$ is in fact a covering map of degree $d=\#V\Gamma$. Thus, $(\Gamma,x_0)$ represents a subgroup $H$ of index $d$ in $F_2=F(a,b)$.

For $i=0,1,\dots,d-1$, we denote by $x_i$ the vertex of the $d$-cycle $\Delta_1$ labelled by $a^d$ at distance $i$ from $x_0$ along $\Delta_1$ in the direction of the flow of $\Delta_1$. That is, $x_i$ is the endpoint of the path in $\Gamma$ labelled by $a^i$.

Consider a maximal tree $T$ in $\Gamma$ consisting of the cycle $\Delta_1$ with the last edge removed. For the dual basis $S_T$ of $\pi_1(\Gamma,x_0)$ the corresponding basis $Z=\mu(S_T)$ of $H$ is $z_0, z_1,\dots, z_d$ where $z_0=a^d$,  $z_i=a^iba^{-(i-1)}$ for $i=1,2,\dots, d-1$, and $z_d=ba^{-(d-1)}$.

Note that $y_i=z_iz_{i-1}\dots z_2z_1=a^iba^{-(i-1)} a^{i-1}ba^{-(i-2)} \dots ab=a^ib^i$ for $i=1,\dots, d-1$. Replacing $z_1,\dots, z_{d-1}$ by $y_1,\dots, y_{d-1}$ in $Z$ corresponds to a sequence of Nielsen transformations and therefore $Y'=z_0, y_1,\dots, y_{d-1}, z_d$ is a free basis of $H$. Note now that $z_dy_{d-1}=ba^{-(d-1)}a^{d-1}b^{d-1}=b^d$. Replacing the element $z_d$ in $Y'$ by $z_dy_{d-1}=b^d$ is a Nielsen transformation which produces a free basis $Y=\{y_0,y_1,\dots, y_d\}$ of $H$ with $y_0=a^d$, $y_d=b^d$, and $y_i=a^ib^i$ for $i=1,\dots, d-1$, and part (1) of the lemma is established.

Now consider a surjective homomorphism $\phi:F(a,b)\to \mathbb Z_d$ given by $\phi(a)=[1]_d$ and $\phi(b)=[-1]_d$. Then $\phi(a^d)=\phi(b^d)=\phi(a^ib^i)=[0]_d$, where $i=1,\dots, d-1$. Therefore, by part (1) of the lemma, $H\le \ker(\phi)$. Since both $H$ and $\ker(\phi)$ have index $d$ in $F(a,b)$, it follows that $H=\ker(\phi)$, and part (2) of the lemma is verified.
\end{proof}

We are now ready to state and prove the first of our main results: an upper bound for $d_{prim}(w_n)$.

\begin{theorem}\label{thm1}
There exists a constant $C'\ge 0$ such that for all integers $n\ge 1$,
\[
			d_{prim}(a^n b^n; F_2)\le \log(n)+C'.
\]
\end{theorem}

\begin{proof}
Let $n\ge 1$ and consider the word $w_n=a^n b^n\in F_2=F(a,b)$. 

Let $d=d(n)\ge 2$ be the smallest integer that does not divide $n$. Lemma \ref{lem1} implies that $2\le d<n$. Also, by Lemma \ref{lem2}, we have $d\le \log(n)+C'$. Express $n$ as $n=kd+r$ where $k\ge1$ and $0<r<d$ is the remainder.

Let $H$ be the subgroup of index $d$ in $F_2$ provided by Lemma~\ref{lem:1} with the free basis $Y=\{y_0,y_1,\dots, y_d\}$, where $y_0=a^d$, $y_d=b^d$ and $y_i=a^ib^i$ for $i=1,\dots, d-1$. 

Note that 

\[
w_n=a^nb^n=(a^d)^ka^rb^r(b^d)^k=y_0^ky_ry_d^k\in H.
\]

The element $w_n=y_0^ky_ry_d^k$ in primitive in $H=F(Y)$ because the generator $y_r$ of $H$ occurs exactly once in $w_n$. Hence,
\[
d_{prim}(w_n;F_2)\le d \le \log(n)+C',
\]
as claimed.

\end{proof}

\begin{figure}
    \centering
    
    \includegraphics{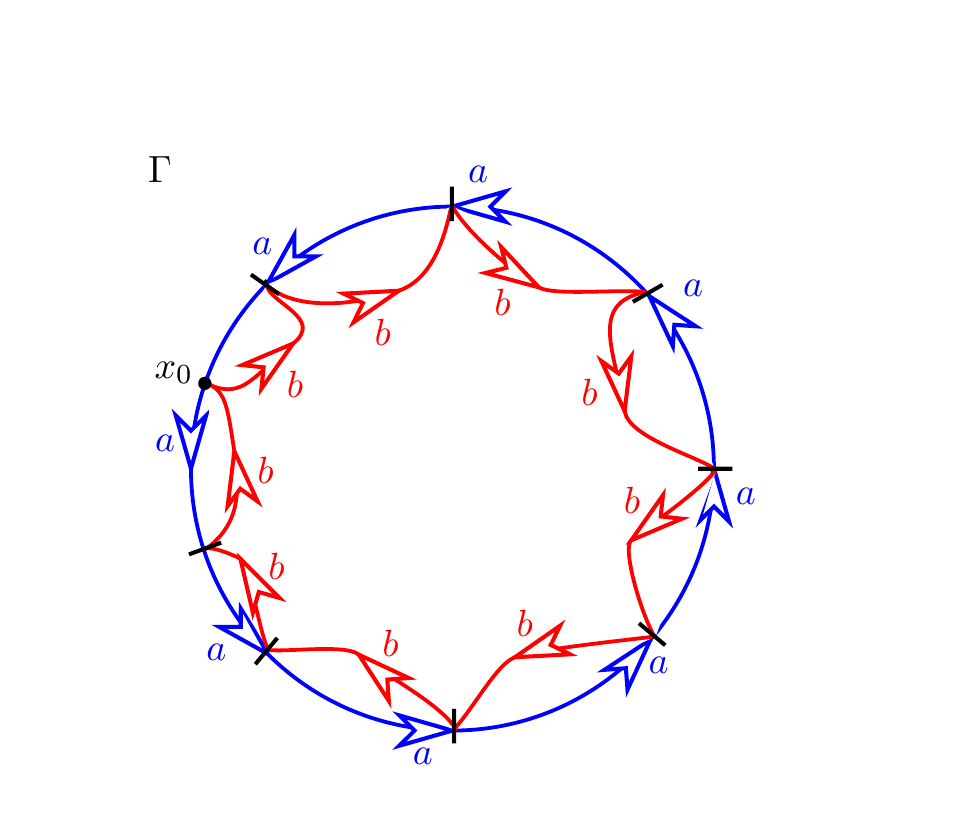}
    
    \caption{The graph $\Gamma$}
    \label{fig2}
\end{figure}

\begin{figure}
    \centering
    
    \includegraphics{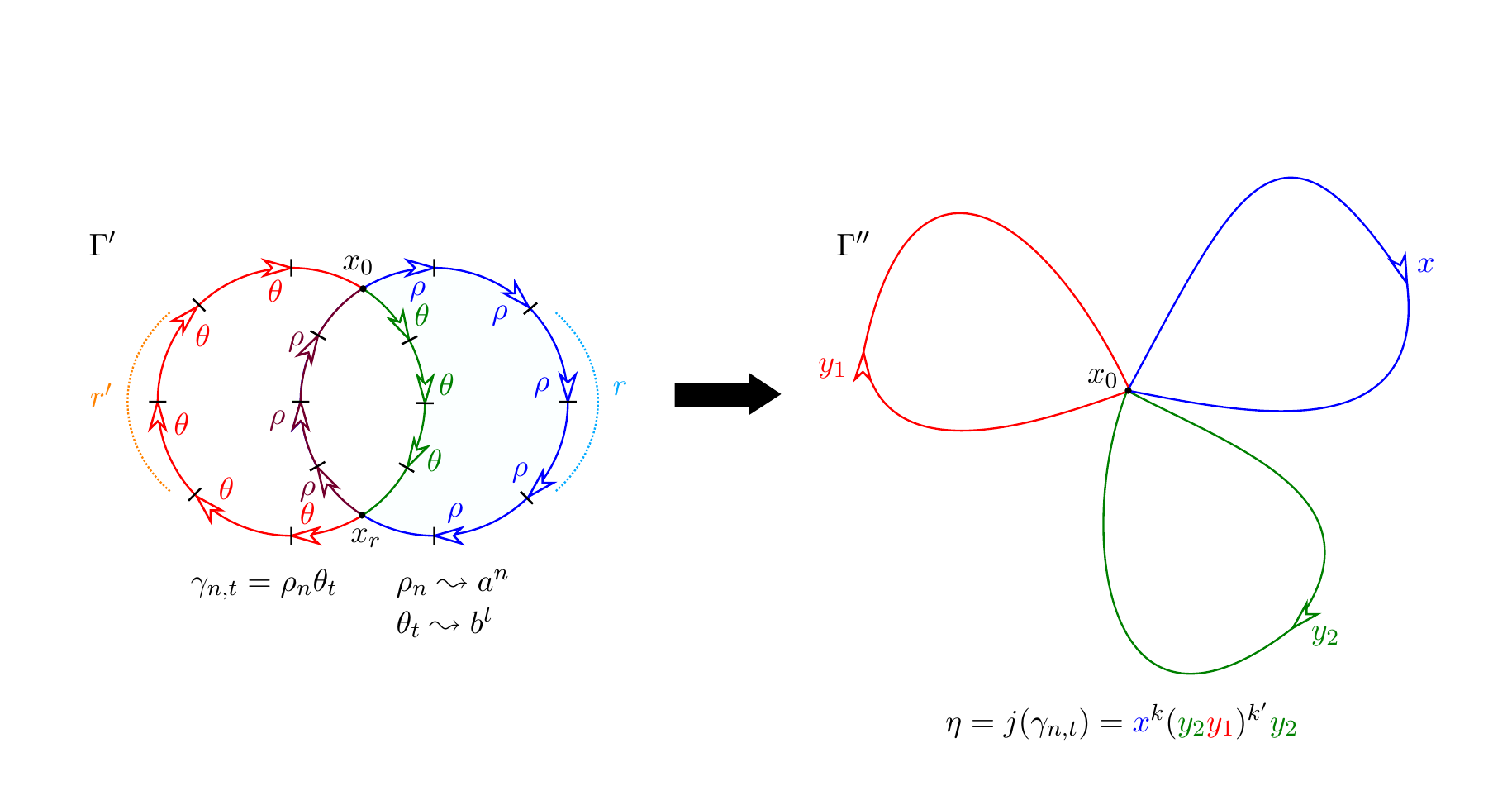}
    
    \caption{$\pi_1(\Gamma'',x_0)=F(x,y_1,y_2)\cong \pi_1(\Gamma',x_0)\cong \pi_1(\Gamma,x_0)$}
    \label{fig3}
\end{figure}

\begin{proposition}\label{prop4}
Let $n,t\ge 1$ and let $d,d'\ge 2$ be integers such that $d\nmid n$ and $d'\nmid t$, and that $d\le n, d'\le t$. Then  
\[
d_{simp}(a^n b^t; F_2)\le d_{prim}(a^n b^t; F_2)\le d+d'-2.
\]
\end{proposition}

\begin{proof}
Let $n,t\ge 1$ and consider the word $w_{n,t}=a^n b^t\in F_2=F(a,b)$. Recall that $d,d'\ge 2$ are integers such that $d\nmid n$, $d'\nmid t$ and that $d\le n, d'\le t$.

Since $d\nmid n$ and $d'\nmid t$, we then have $2\le d<n$, $2\le d'<t$. 

First, divide $n$ with remainder by $d$, and $t$ with remainder by $d'$. Thus, express $n$ and $t$ as $n=kd+r$ and $t=k'd'+r'$ where $k,k'\ge 0$ and $0<r<d$, $0<r'<d'$, respectively. Note that since $2\le d<n$ and  $2\le d'<t$, we actually have $k, k'\ge 1$. 

Let $\Delta_1$ be the simplicial cycle of length $d$ given from the $A$-graph structure by labeling it as an $a^d$-cycle. Similarly, let $\Delta_2$ be the simplicial cycle of length $d'$ endowed with the $A$-graph structure by labeling it as a $b^{d'}$-cycle. We pick a base-vertex $x_0$ on $\Delta_1$ and a base-vertex $z_0$ on $\Delta_2$. On the graph $\Delta_1$, let $x_r$ be a vertex such that the arc along $\Delta_1$ from $x_0$ labelled by $a^r$ ends at $x_r$. Let $z_{d'-r'}$ be the vertex of $\Delta_2$ such that the arc along $\Delta_2$ starting at $z_{d'-r'}$ and labelled by $b^{r'}$ ends at $z_0$. 

Next, identify $x_0$ with $z_0$, and $x_r$ with $z_{d'-r'}$.  Denote the resulting graph by $\Gamma'$; see Figure \ref{fig3}. In the graph $\Gamma'$, we still denote the image of the vertex $x_0$ (and $z_0$) by $x_0$, and we denote the image of the vertex $x_r$ by $x_r$.  
 
Thus, $\Gamma'$ is a connected folded $A$-graph with $\#V\Gamma'=d+d'-2$.
 
By a standard Marshall Hall Theorem proof argument, see \cite[Lemma~8.10]{KM02}, there exists a finite connected folded $A$-graph $\Gamma$ such that $V\Gamma=V\Gamma'$ and such that $\Gamma$ is a finite cover of $R_2$ of degree equal to  $\#V\Gamma=\#V\Gamma'=d+d'-2$. Let $H\le F_2=F(a,b)$ be the subgroup of index $d+d'-2$ corresponding to the cover $(\Gamma,x_0)\to (R_2,v_0)$.

Let $\rho_n$ be the path in $\Gamma'$ starting at $x_0$ and labelled by $a^n$. Since $n=kd+r$, the path $\rho_n$ ends at the vertex $x_r$.
Now let $\theta_t$ be the path in $\Gamma'$ starting at $x_r$ and labelled by $b^t$. Since $t=k'd'+r'$, the construction of $\Gamma'$ implies that $\theta_t$ ends at $x_0$. Thus, the path $\gamma_{n,t}=\rho_n\theta_t$ in $\Gamma'$ has label $w_{n,t}=a^nb^t$ and begins and ends at $x_0$.

Consider the subarc $\beta$ of $\rho_n$ from $x_0$ to $x_r$, labelled by $a^{r}$. The arc $\beta$ is shown in the blue in  Figure \ref{fig3}.

Notice that $T=\beta$ is a subtree of $\Gamma'$. Let $\Gamma''$ be the graph obtained from $\Gamma'$ by collapsing the arc $\beta$ to a point, where the image of the vertex $x_0$ in $\Gamma''$ will still be called $x_0$. Then the image of $\rho_n$ in $\Gamma'$ is a loop $x$ at $x_0$, the image of $\theta_t$ from $b$-edges, red and green, consists of two topological loops at $x_0$, denoted  $y_1$ and $y_2$,  respectively.

Let $j: \Gamma'\to\Gamma''$ be the map given by collapsing $\beta$. Thus, $j$ is a homotopy equivalence and $j_\#:\pi_1(\Gamma',x_0)\to\pi_1(\Gamma'',x_0)$ is an isomorphism. Then $\pi_1(\Gamma'',x_0)=F(x,y_1,y_2)\cong \pi_1(\Gamma'',x_0)$, and $S=\{x,y_1,y_2\}$ is a free basis of $\pi_1(\Gamma'',x_0)$.

For the loop $\gamma_{n,t}=\rho_{n}\theta_{t}$ at $x_0$ in $\Gamma'$ labelled by $a^{kd+r}b^{k'd'+r'}$, the collapsing map $j$ produces a closed path \[\eta=j(\gamma_{n,t})=x^{k} (y_2 y_1)^{k'} y_2\in F(x,y_1,y_2)=\pi_1(\Gamma'',x_0).\]

Now consider an automorphism $\varphi$ of the free group $F(x,y_1,y_2)$ defined by 
\[
\varphi: \{x \mapsto x, y_1 \mapsto y_2^{-1} y_1, y_2 \mapsto y_2 \}.
\]

Then $\varphi(x)=x$ and $\varphi(y_2 y_1)=y_2 y_2^{-1} y_1=y_1$ imply that \[\varphi(\eta)=\varphi(x^{k} (y_2 y_1)^{k'} y_2)=x^{k} y_1^{k'} y_2.\]

The freely reduced word $\varphi(\eta)$ is primitive in $F(x,y_1,y_2)$ since the letter $y_2^{\pm 1}$ occurs precisely once in this word. Then since $\varphi$ was an automorphism of $F(x,y_1,y_2)$, it follows that $\eta\in F(x,y_1,y_2)$ is primitive too. Thus, $\gamma_{n,t}$ is primitive in $\pi_1(\Gamma',x_0)$ since \[\pi_1(\Gamma'',x_0)=F(x,y_1,y_2)\cong \pi_1(\Gamma'',x_0)\cong \pi_1(\Gamma',x_0)\cong \pi_1(\Gamma,x_0).\] 

Since $\Gamma'$ is a subgraph of $\Gamma$, it follows that $\gamma_{n,t}$ is primitive in $\pi_1(\Gamma,x_0)$ as well. Since the covering map $\Gamma\to R_2$ induces an isomorphism $\pi_1(\Gamma,x_0)\cong H$ sending $\gamma_{n,t}$ to $w_{n,t}$, we conclude that $w_{n,t}\in H$ is primitive in $H$ as well. Since, as noted above, $[F_2:H]=\#V\Gamma=\#V\Gamma'=d+d'-2$, the definition of the primitivity index now implies that 
\[
d_{simp}(w_{n,t}; F_2)\le d_{prim}(w_{n,t}; F_2)\le d+d'-2,
\]
as required.
\end{proof}

\begin{corollary}\label{cor1}
Let $C'\ge 0$ and $n_0\ge 3$ be the constants from Lemma \ref{lem2}. Then we have the following:
\begin{enumerate}
\item[(a)] For all $n,t\ge n_0$, $d_{simp}(a^n b^t; F_2)\le d_{prim}(a^n b^t; F_2)\le \log(n)+\log(t)+2\log(2)$.
\item[(b)] For all $n,t\ge 3$, $d_{simp}(a^n b^t; F_2)\le d_{prim}(a^n b^t; F_2)\le \log(n)+\log(t)+2C'-2$.
\end{enumerate}
\end{corollary}

\begin{proof}
Let $n,t\ge 3$ be integers. Put $d=d(n)$ and $d'=d(t)$. Thus, $d,d'\ge 2$ are the smallest integers not dividing $n$ and $t$ accordingly.

By Lemma~\ref{lem1}, we have $1<d<n$ and $1<d'<t$.

By Lemma~\ref{lem2}, we have $d\le \log(n)+C'$ and $d'\le \log(t)+C'$. Hence, Proposition~\ref{prop4} implies that
\[
d_{simp}(a^n b^t; F_2)\le d_{prim}(a^n b^t; F_2)\le d+d'-2\le \log(n)+\log(t)+2C'-2,
\]
and part (b) holds, as required.

Suppose further that $n,t\ge n_0$. Then by Lemma~\ref{lem2}, we have $d\le \log(n)+\log(2)+1$ and $d'\le \log(t)+\log(2)+1$. Hence, by Proposition \ref{prop4}, we have

\[
d_{simp}(a^n b^t; F_2)\le d_{prim}(a^n b^t; F_2)\le d+d'-2\le \log(n)+\log(t)+2\log(2),
\]
and part (a) of the corollary holds, as required.
\end{proof}

The following lemma, suggested to us by the referee, provides a tool for estimating from below $d_{prim}(a^{n_i}b^{n_i}; F_2)$ in part (b) of Theorem~\ref{t:main2} that is simpler than our original approach.

\begin{lemma}\label{power} 
Let $H\le F_2$ be a subgroup of finite index $p\ge 1$. Let $1\le k\le p, 1\le l\le p$ be smallest positive integers such that $a^k\in H$ and $b^l\in H$. Then there exists a free basis $Y$ for $H$ such that $a^k, b^l\in Y$.
\end{lemma}

\begin{proof}
Let $(\Gamma,x_0)$ be the Stallings subgroup graph representing $H$, that is, a $p$-fold cover of the 2-rose $R_2$ corresponding to $H$.
Let $\gamma_a$ be the closed path at $x_0$ in $\Gamma$ labelled by $a^k$, and let $\gamma_b$ be the closed path the closed path at $x_0$ in $\Gamma$ labelled by $b^l$. The minimality of choices of $k$ and $l$ implies that both $\gamma_a$ and $\gamma_b$ are simple closed circuits in $\Gamma$. 

There are two cases two consider. 

{\it Case 1.} The paths $\gamma_a$ and $\gamma_b$ share no vertices in common except for $x_0$.

In this case the subgraph $\Gamma'$ of $\Gamma$ spanned by $\gamma_a,\gamma_b$ is a wedge of two circles, labelled by $a^k$ and $b^l$. It is obvious that $\gamma_a,\gamma_b$ is a free basis of $\pi_1(\Gamma,x_0)$. We can extend this free basis to a free basis of $\pi_1(\Gamma',x_0)$ since $\Gamma'$ is a connected subgraph of $\Gamma$ and $\pi_1(\Gamma',x_0)$ is a free factor of $\pi_1(\Gamma,x_0)$. Therefore, the conclusion of of the lemma holds in this case.

{\it Case 2.} The paths $\gamma_a$ and $\gamma_b$ have at least one other vertex in common apart from $x_0$.

Let $x_0, x_1, x_2, \dots, x_m=x_0$ be all  the vertices of $\gamma_a$ in the order they occur along $\gamma_b$. Denote by $a^{s_i}$ the label of the segment of $\gamma_a$ from $x_0$ to $x_i$, and denote by $b^{t_i}$ the label of the segment of $\gamma_b$ from $x_{i-1}$ to $x_i$, where $i=1,\dots, m$.  By our construction we have $t_i>0$ and $s_{i-1}\ne s_i$ but we are not guaranteed that $s_{i-1}<s_i$.

Choose a maximal subtree $T$ of $\Gamma$ which contains the initial segment of $\gamma_a$ labeled by $a^{k-1}$. Let $Z$ be the free basis of $H$ dual to this maximal tree. The last edge of $\gamma_a$ labeled by $a$ shows that $z_0=a^k\in Z$.  

The arc of $\gamma_b$ from $x_0$ to $x_1$ produces the element $z_1=b^{t_1}a^{-s_1}\in Z$. The arc of $\gamma_b$ from $x_1$ to $x_2$ produces the element $z_2=a^{s_1}b^{t_2}a^{-s_2}\in Z$. The arc of $\gamma_b$ from $x_2$ to $x_3$ produces the element $z_3=a^{s_2}b^{t_3}a^{-s_3}\in Z$. Thus, for $i=2,\dots, m-1$ we get an element $z_{i}=a^{s_{i-1}}b^{t_{i}}a^{-s_i}\in Z$. Finally, the arc from $x_m$ to $x_0$ in $\gamma_b$ produces the element $z_m=a^{s_m}b^{t_m}\in Z$. Then 
\[
y=z_1z_2\dots z_m=b^{t_1}a^{-s_1} a^{s_1}b^{t_2}a^{-s_2} a^{s_2}b^{t_3}a^{-s_3}\dots a^{s_{m-1}}b^{t_{i}}a^{-s_m} a^{s_m}b^{t_m}=b^{t_1+\dots +t_m}=b^l.
\]

Replacing $z_m$ by $y=z_1z_2\dots z_m=b^l$ in $Z$ is a Nielsen transformation which produces a free basis $Y$ of $H$ that contains both $a^k$ and $b^l$, as required.

\end{proof}

We now proceed by finding a lower bound in the following theorem:

\begin{theorem}\label{thm2}
 For all integers $i\ge 3$, if $n_i=lcm(1,2,\dots,i)$, we have:
\[
d_{prim}(a^{n_i} b^{n_i}; F_2)\ge d(n_i)\ge \log(n_i)-o(\log(n_i)).
\]
\end{theorem}

\begin{proof}
Let $w_{n_i}=a^{n_i}b^{n_i}$ where $i\ge 3$. Put $d_i=d(n_i)$. Let $H\le F_2$ be a subgroup of finite index such that $w_{n_i}\in H$ and that $[F_2:H]=m<d_i$.
Let $1\le k\le m, 1\le l\le m$ be smallest positive integers such that $a^k\in H$ and $b^l\in H$. By Lemma~\ref{power} there exists a free basis $Y$ for $H$ such that $y_1=a^k$ and $y_2=b^l$ belong to $Y$. Since $k,l\le m<d(n_i)$, we have $k|n_i$ and $l\vert n_i$. Recall also that by Lemma~\ref{lem1} $1<d(n_i)<n_i$. Hence, $n=pk=ql$ with $p,q\ge 2$. Then in the basis $Y$ we have $w_{n_i}=y_1^py_2^q$. Since $p,q\ge 2$, the element $y_1^py_2^q$ is not primitive in $F(y_1,y_2)$ by Corollary~\ref{cor2}. Therefore, $w_{n_i}=y_1^py_2^q$ is not primitive in $H=F(Y)$ by Proposition~\ref{p:pr}.

The definition of the the primitivity index now implies that $d_{prim}(w_{n_i}; F(a,b))\ge d(n_i)$. Finally, Lemma~\ref{lem2} implies that $d(n_i)\ge \log(n_i)-o(\log(n_i))$.

\end{proof}

Now that we have established upper and lower bounds, Theorem \ref{thm1} and Theorem \ref{thm2} together imply the following corollary:

\begin{corollary}\label{cor3}
There exists a constant $C\ge 0$ such that for all integers $i\ge 1$, if $n_i=lcm(1,2,\dots,i)$, then
\[
\log(n_i)-o(\log(n_i))\le d_{prim}(a^{n_i} b^{n_i}; F_2)\le \log(n_i)+C.
\]
\end{corollary}

\begin{theorem}\label{thm4}
For all integers $n\ge 2$, we have
\[
d_{simp}(a^n b^n; F_2)=2.
\]
\end{theorem}

\begin{proof}
Let $n\ge 2$. Consider the word $w_n=a^n b^n\in F_2=F(a,b)$. In order to prove the result, we will break into two cases:

\begin{flushleft}
\emph{Case 1} ($n$ is odd):
\end{flushleft}

Put $d=d'=2$. Since $n\ge 2$ is odd, we have $n\ge 3$ and $d<n$. Thus, since $d\nmid n$, Proposition~\ref{prop4} implies that  $d_{simp}(a^n b^n; F_2)\le d+d'-2=2$.

\begin{flushleft}
\emph{Case 2} ($n$ is even):
\end{flushleft}

Thus, $n=2k$ for an integer $k\ge 1$.

For $d=2$ let $H$ be the subgroup of index $2$ in $F_2=F(a,b)$, provided by Lemma~\ref{lem:1}, with the free basis $Y=\{y_0,y_1,y_2\}$ where $y_0=a^2$, $y_1=ab$, and $y_2=b^2$. 

We have

\[
w_n=a^nb^n=(a^2)^k(b^2)^k=y_0^ky_2^k\in H.
\]
The element $w_n= y_0^ky_2^k$ is simple in $H$ since $w_n$ does not involve the generator $y_1$ and thus $w_n$ belongs to a proper free factor $F(y_0,y_2)$ of $H=F(Y)$. Hence, $d_{prim}(w_n; F_2)\le 2$. 

Finally, Corollary~\ref{cor2} implies that for $n\ge 2$ the element $w_n=a^nb^n$ is not simple in $F(a,b)$, so that $d_{simp}(a^n b^n; F_2)\ne 1$. Hence, for every $n\ge 2$, $d_{simp}(a^n b^n; F_2)=2$, as required.
\end{proof}

\begin{remark}

One can alternatively argue in Case~1 in the above proof by relying on Lemma~\ref{lem:1} instead of Proposition~\ref{prop4}. Indeed, take $d=2$, and let $H$ be the subgroup of index $2$ in $F_2=F(a,b)$ provided by Lemma~\ref{lem:1} with the free basis $Y=\{y_0,y_1,y_2\}$ where $y_0=a^2$, $y_1=ab$, and $y_2=b^2$. Then $n=2k+1$ for $k\ge 1$ and $w_n=a^nb^n=(a^2)^kab(b^2)^k=y_0^ky_1y_2^k\in H$ is primitive in $H$ because it involves the generator $y_1$ exactly once. Hence, $d_{simp}(w_n;F_2)\le d_{prim}(w_n; F_2)\le 2$.
\end{remark}
\nocite{*}

\bibliographystyle{amsplain}

\bibliography{foo}

\end{document}